\documentclass[preprint,12pt]{elsarticle}

\usepackage{amsmath, amsthm, amssymb}

\newtheorem{thm}{Theorem}[section]
\newtheorem{cor}[thm]{Corollary}
\newtheorem{lem}[thm]{Lemma}

\newtheorem{ques}[thm]{Open Question}

\theoremstyle{remark}

\newtheorem{expl}[thm]{Example}

\theoremstyle{definition}
\newtheorem{defn}[thm]{Definition}

\newcommand{\dom}{\mathrm{dom} }
\newcommand{\tp}{\mathrm{tp} }
\newcommand{\qftp}{\mathrm{qftp} }

\newcommand{\ID}{\mathrm{ID} }

\newcommand{\leng}{\mathrm{lg} }

\newcommand{\nlhd}{\not\hspace{-5pt}\lhd\hspace{4pt} }
\newcommand{\Maj}{\mathrm{Maj} }
\newcommand{\Lev}{\mathrm{Lev} }

\journal{Annals of Pure and Applied Logic}

\begin{document}

\begin{frontmatter}

\title{Definability of types over finite partial order indiscernibles}

\author{Vincent Guingona}
\address{University of Notre Dame \\ Department of Mathematics \\ 255 Hurley, Notre Dame, IN 46556}
\ead{guingona.1@nd.edu}
\ead[url]{http://www.nd.edu/~vguingon/}

\begin{abstract}
 In this paper, we show that a partitioned formula $\varphi$ is dependent if and only if $\varphi$ has uniform definability of types over finite partial order indiscernibles.  This generalizes our result from a previous paper \cite{Mypaper2}.  We show this by giving a decomposition of the truth values of an externally definable formula on a finite partial order indiscernible.
\end{abstract}

\begin{keyword}
 definability \sep types \sep dependent \sep NIP \sep generalized indiscernibles
 
 \MSC[2010] 03C45
\end{keyword}

\end{frontmatter}

\section{Introduction}\label{Section_Intro}


In \cite{Mypaper2}, we introduce the notion of uniform definability of types over finite sets (UDTFS) and conjecture that all dependent formulas have UDTFS (we call this the UDTFS Conjecture).  In that paper, we approach a solution to the conjecture from two distinct directions.  First, we take a subclass of the class of dependent theories and show that this subclass has UDTFS; namely, we show that all dp-minimal theories have UDTFS.  We hope to show this for larger subclasses in future papers.  Our second approach involves slightly weakening the definition of UDTFS.  In the first section of \cite{Mypaper2}, we actually give a characterization of dependent formulas in terms of definability of types.  Theorem 1.2 (ii) of \cite{Mypaper2} states that a formula is dependent if and only if it has uniform definability of types over finite indiscernible sequences.

Indiscernible sequences are very strong and well behaved in the context of dependent theories, so this result is not too surprising.  On the other hand, as one continues to weaken the assumption of ``indiscernible sequence,'' one gets closer to solving the UDTFS Conjecture.  In this paper, we generalize Theorem 1.2 (ii) of \cite{Mypaper2} using generalized indiscernible sequences.  We prove the following theorem:

\begin{thm}\label{Thm_UDTFIPOS}
 The following are equivalent for a partitioned formula $\varphi(\overline{x}; \overline{y})$:
 \begin{itemize}
  \item [(i)] $\varphi$ is dependent;
  \item [(ii)] there exists a formula $\psi(\overline{y}; \overline{z}_0, ..., \overline{z}_{n-1})$ such that, for all finite partial orders $(P; \unlhd)$, all (generalized) indiscernibles $\langle \overline{b}_i : i \in P \rangle$ (see Definition \ref{Defn_GenIndiscSequence} below), and all types $p \in S_\varphi( \{ \overline{b}_i : i \in P \} )$, there exists $i_0, ..., i_{n-1} \in P$ so that, for all $i \in P$, $\varphi(\overline{x}; \overline{b}_i) \in p(\overline{x})$ if and only if $\models \psi(\overline{b}_i; \overline{b}_{i_0}, ..., \overline{b}_{i_{n-1}})$.
 \end{itemize}
\end{thm}

That is, we show that a formula $\varphi$ is dependent if and only if it has uniform definability of types over finite partial order indiscernibles.  The notion of generalized indiscernibles is first introduced in Chapter VII of \cite{Shelah}.  As in the work of Scow \cite{Scow}, this paper characterizes dependence in terms of generalized indiscernible sequences.  However, in this paper, we use partial order indiscernibles instead of ordered graph indiscernibles.

If we can push this to its natural conclusion, we could solve the UDTFS Conjecture.  For example, suppose that $\varphi$ has independence dimension $\le n$ and we took as our index language $S = \{ P_\eta : \eta \in {}^{n+1} 2 \}$ for $(n+1)$-ary predicates $P_\eta$.  Then $\varphi$ has UDTFS if and only if it has uniform definability of types over finite $S$-structure indiscernibles.  Thus, we view Theorem \ref{Thm_UDTFIPOS} as a definite step toward solving the UDTFS Conjecture.

For this paper, a ``formula'' will mean a $\emptyset$-definable formula in a fixed language $L$ unless otherwise specified.  If $\theta(\overline{x})$ is a formula, then let me denote $\theta(\overline{x})^0 = \neg \theta(\overline{x})$ and $\theta(\overline{x})^1 = \theta(\overline{x})$.  We will be working in a complete, first-order theory $T$ in a fixed language $L$ with monster model $\mathfrak{C}$.  Fix $M \models T$ (so $M \preceq \mathfrak{C}$) and a partitioned $L$-formula $\varphi(\overline{x}; \overline{y})$.  By $\varphi(M; \overline{b})$ for some $\overline{b} \in \mathfrak{C}^{\leng(\overline{y})}$, we mean the following subset of $M^{\leng(\overline{x})}$:
\[
 \varphi(M; \overline{b}) = \{ \overline{a} \in M^{\leng(\overline{x})} : \models \varphi(\overline{a}; \overline{b}) \}.
\]
We will say that a set $B \subseteq \mathfrak{C}^{\leng(\overline{y})}$ is $\varphi$-\textit{independent} if, for any map $s : B \rightarrow 2$, the set of formulas $\{ \varphi(\overline{x}; \overline{b})^{s(\overline{b})} : \overline{b} \in B \}$ is consistent.  We will say that $\varphi$ has \textit{independence dimension} $N < \omega$, which we will denote by $\ID(\varphi) = N$, if $N$ is maximal such that there exists $B \subseteq \mathfrak{C}^{\leng(\overline{y})}$ with $|B| = N$ where $B$ is $\varphi$-independent.  We will say that $\varphi$ is dependent (some authors call this NIP for ``not the independence property'') if $\ID(\varphi) = N$ for some $N < \omega$.  Finally, we will say that a theory $T$ is dependent if all partitioned formulas are dependent.

Fix a set of partitioned formulas $\Phi(\overline{x}; \overline{y}) = \{ \varphi_i(\overline{x}; \overline{y}) : i \in I \}$.  By a ``$\Phi$-type over $B$'' for some set $B$ of $\leng(\overline{y})$-tuples we mean a consistent set of formulas of the form $\varphi_i(\overline{x}; \overline{b})^t$ for some $t < 2$ and ranging over all $\overline{b} \in B$ and $i \in I$.  If $p$ is a $\Phi$-type over $B$, then we will say that $p$ has domain $\dom(p) = B$.  For any $B$ a set of $\leng(\overline{y})$-tuples, the space of all $\Phi$-types with domain $B$ is denoted $S_\Phi(B)$.  If $\Phi = \{ \varphi \}$ is a singleton, then we will replace $\varphi$ with $\{ \varphi \}$ in our previous definitions.

The remainder of this paper is organized as follows: In Section \ref{Section_PO}, we focus only on partial orders.  Abstractifying the notion of indiscernibility to simply colorings on partial orders, we produce a means of partitioning the ordering into homogeneous subsets with respect to the coloring.  Applying this to indiscernibles, this generalizes the ``bounded alternation rank'' characterization of dependent formulas (Theorem II.4.13 (2) of \cite{Shelah}) and may be of independent interest.  In Section \ref{Section_GeneralIndiscn}, we define and discuss partial order indiscernibles.  We prove Theorem \ref{Thm_UDTFIPOS} above using the techniques of \cite{Mypaper2} and the partitioning theorem from Section \ref{Section_PO}.  Finally, in Section \ref{Section_Discussion}, we discuss the broader implications of this result and state some natural open questions that remain.

\section{Partial Orders}\label{Section_PO}


Before discussing general indiscernibility and even model theory, we first work in the universe of pure partial orders.  The discussion before Dilworth's Theorem (Theorem \ref{Thm_BoundAntiChainBoundChain}) is elementary and the results are certainly not due to this author.

Fix $(P; \unlhd)$ a partial order.  An \emph{antichain} $A$ of the partial order $P$ is a subset of $P$ such that, for all $i, j \in A$, $i \nlhd j$ and $j \nlhd i$.  By contrast, a \emph{chain} $C$ of the partial order $P$ is a subset of $P$ such that, for all $i, j \in C$, $i \unlhd j$ or $j \unlhd i$.  We will say that an antichain $A$ is \emph{maximal} if there does not exist $i \in P-A$ such that $A \cup \{ i \}$ is an antichain, and we similarly define a \emph{maximal} chain.  For any subset $P_0 \subseteq P$, $(P_0; \unlhd |_{P_0 \times P_0} )$ is a partial order and will be called a \emph{suborder} of $(P; \unlhd)$.  For any maximal antichain $A$, define the following sets:
\begin{itemize}
 \item [(i)] $D(A) = \{ i \in P : (\exists j \in A)(i \lhd j) \}$ (the downward closure of $A$).
 \item [(ii)] $U(A) = \{ i \in P : (\exists j \in A)(j \lhd i) \}$ (the upward closure of $A$).
\end{itemize}

\begin{lem}\label{Lem_AntichainLemma}
 Fix $A \subseteq P$ a maximal antichain.  For any $i \in P-A$, exactly one of the following hold:
 \begin{itemize}
  \item [(i)] There exists $j \in A$ such that $i \lhd j$.
  \item [(ii)] There exists $j \in A$ such that $j \lhd i$.
 \end{itemize}
 That is, $\{ D(A), A, U(A) \}$ is a partition of $P$.
\end{lem}

\begin{proof}
 By the maximality of $A$ and transitivity.
\end{proof}

In general, for any antichain $A$, if $j \lhd i$ for some $i \in A$, then $i \nlhd j$ for all $i \in A$.  In this case, we say that $j \lhd A$.  If there exists $i \in A$ such that $i \lhd j$, we say that $A \lhd j$.  If $j \nlhd A$ and $A \nlhd j$, then $A \cup \{ j \}$ is again an antichain.  If $j \lhd A$ or $j \in A$, we write $j \unlhd A$ and similarly for $A \unlhd j$.

\begin{lem}\label{Lem_AntichainContainment}
 Fix $A \subseteq P$ a maximal antichain and suppose $A' \subseteq (D(A) \cup A)$ is an antichain.  Then there exists $A''$ with $A' \subseteq A'' \subseteq (D(A) \cup A)$ such that $A''$ is a maximal antichain of $P$ (the whole order).
\end{lem}

\begin{proof}
 If $A'$ is not a maximal antichain of $P$, then there exists $j \in P - A'$ such that $A' \cup \{ j \}$ is an antichain.  If $j \notin (D(A) \cup A)$, then $j \in U(A)$ by Lemma \ref{Lem_AntichainLemma}.  Therefore, there exists $i \in A$ so that $i \lhd j$.  Since $A' \subseteq (D(A) \cup A)$, either $i \in A'$, $A' \lhd i$, or $A' \cup \{ i \}$ is an antichain.  If $i \in A'$ or $A' \lhd i$, then $A' \lhd j$, contrary to assumption.  Therefore, $A' \cup \{ i \}$ is an antichain contained in $(D(A) \cup A)$ that is $\lhd$-below $j$.  Use Zorn's Lemma to conclude.
\end{proof}

This also holds for $U(A) \cup A$ by symmetry.  Given two maximal antichains $A, A' \subseteq P$, say that $A \unlhd A'$ if $A \subseteq D(A') \cup A'$ (i.e., for all $i \in A$, $i \in A'$ or $i \lhd A'$).  Of course, there can be maximal antichains that are incomparable, but transitivity will clearly hold for this relation.  Furthermore, for any maximal antichain $A \subseteq P$, if $D(A) \neq \emptyset$, then there exists $A' \subseteq P$ a maximal antichain such that $A' \lhd A$ (and similarly if $U(A) \neq \emptyset$) by Lemma \ref{Lem_AntichainContainment}.

For any $A \unlhd A'$ maximal antichains of $P$, define $(A, A') = U(A) \cap D(A')$, let $[A,A') = (A \cup U(A)) \cap D(A')$, and let $[A,A'] = A \cup (A,A') \cup A'$.  Let $[-\infty, A) = D(A)$, let $[-\infty, A] = A \cup D(A)$, let $[A, \infty) = A \cup U(A)$, and let $[-\infty, \infty) = P$ (think of these as ``intervals'' of $P$).  So, for any $A_0 \lhd A_1 \lhd ... \lhd A_n$ maximal antichains of $P$, $[-\infty, A_0)$, $[A_0, A_1)$, $...$, $[A_{n-1}, A_n)$, $[A_n, \infty)$ is a partition of $P$.

We define $\Lev_n^-(P)$, the $n$th level of $P$ from below, by induction as follows:
\[
 \Lev_n^-(P) = \left\{ i \in P - \bigcup_{\ell < n} \Lev_\ell^-(P) : \left( \nexists j \in P - \bigcup_{\ell < n} \Lev_\ell^-(P) \right)(j \lhd i) \right\}.
\]
So $\Lev_0^-(P)$ is the antichain of the least elements of $P$, and $\Lev_1^-(P)$ is the antichain of the least elements of $P - \Lev_0^-(P)$, and so on.  We define $\Lev_n^+(P)$ by reversing the ordering.  Notice that, for all $i \in \Lev_n^-(P)$, there exists $i_0 \in \Lev_0^-(P)$, ..., $i_{n-1} \in \Lev_{n-1}^-(P)$ such that $i_0 \lhd ... \lhd i_{n-1} \lhd i$.

\begin{thm}[Dilworth's Theorem, \cite{Dilworth}]\label{Thm_BoundAntiChainBoundChain}
 Fix $n < \omega$.  If $(P; \unlhd)$ is a finite partial order such that, for all antichains $A \subseteq P$, $|A| \le n$, then $P$ is the disjoint union of at most $n$ chains.
\end{thm}

We now discuss $2$-colorings of a finite partial order $(P; \unlhd)$.  We will use this in the next section when proving definability of types over finite partial order indiscernibles.

\begin{defn}\label{Defn_IndiscernibleColoring}
 Fix $(P; \unlhd)$ a partial order, $f : P \rightarrow \{ 0, 1 \}$, and $N < \omega$.  We say that $f$ is a \emph{$N$-indiscernible coloring of $P$} if,
 \begin{itemize}
  \item [(i)] for all antichains $A \subseteq P$, there exists $t < 2$ such that $| \{ i \in P : f(i) = t \} | \le N$; and
  \item [(ii)] there does not exist $i_0 \lhd i_1 \lhd ... \lhd i_{2N+1}$ from $P$ such that, for all $\ell < 2N+1$, $f(i_\ell) = 1 - f(i_{\ell+1})$ (that is, the coloring on any chain does not alternate more than $2N+1$ times).
 \end{itemize}
\end{defn}

Fix $(P; \unlhd)$ a finite partial order and $f : P \rightarrow 2$ a $N$-indiscernible coloring of $P$.  For any subset $X \subseteq P$ and $t < 2$, define $X^t$ as follows:
\[
 X^t = \{ i \in X : f(i) = t \} = (f^{-1}(t) \cap X).
\]
Note that $X = X^0 \cup X^1$ and $X^0 \cap X^1 = \emptyset$.  For any antichain $A \subseteq P$ with $|A| > 2N$, there exists a unique $t < 2$ such that $|A^t| \le N$.  If not, then $A$ would violate condition (i) of Definition \ref{Defn_IndiscernibleColoring}.  In this case, define $\Maj(A) = t$ ($\Maj$ stands for ``majority'').  We now use this to give a means of breaking down partial orders $P$ in terms of subsets $X$ on which $f$ is constant.

\begin{lem}\label{Lem_IndiscBreakdown}
 Let $M = (2N+1)(N+1)$.  There exists $A_0 \lhd ... \lhd A_{K-1}$ for $K \le 2N+2$ maximal antichains of $P$ such that, for all $n \le K$ and all antichains $A \subseteq [A_{n-1}, A_n)$, $|A^{n (\mathrm{mod}\ 2)}| \le M$ (let $A_{-1} = - \infty$ and $A_K = \infty$).  That is, each $P_n = [A_{n-1}, A_n)$ is such that all antichains $A \subseteq P_n$ have $f(i) = n+1 (\mathrm{mod}\ 2)$ for ``almost all'' $i \in A$.
\end{lem}

\begin{proof}
 We inductively construct, for each $n$, $A_n \subseteq P$ a maximal antichain of $P$ as follows: Fix $n \ge 0$ and suppose $A_\ell$ are defined for all $\ell < n$.  If it exists, choose $A_n \subseteq P$ maximal such that
 \begin{itemize}
  \item [(i)] $A_0 \lhd ... \lhd A_{n-1} \lhd A_n$,
  \item [(ii)] $| A_n^{n (\mathrm{mod}\ 2)} | > M$ (hence $\Maj(A_n) \equiv n (\mathrm{mod}\ 2)$), and
  \item [(iii)] $A_n$ is $\lhd$-minimal such.
 \end{itemize}
 If no such $A_n$ exists, set $K = n$ and the construction terminates.  For any $n \le K$ and any antichain $A \subseteq [A_{n-1}, A_n)$, if $|A^{n (\mathrm{mod}\ 2)}| > M$, then extend this to a maximal antichain $A' \subseteq [A_{n-1}, A_n]$ (which exists by Lemma \ref{Lem_AntichainContainment} on the suborder $P_0 = (-\infty, A_n]$).  Then $A_{n-1} \lhd A' \lhd A_n$ and $|(A')^{n (\mathrm{mod}\ 2)}| > M$, contrary to the minimality of $A_n$.  Therefore, $|A^{n (\mathrm{mod}\ 2)}| \le M$.  We now show that this process terminates in $K < 2N+2$ steps.

 Assuming $K \ge 2N+2$, inductively define, for each $n < K$, $A^*_n \subseteq A_n^{n (\mathrm{mod}\ 2)}$ with
 \begin{itemize}
  \item [(i)] $|A^*_n| > (2N+1-n)(N+1)$, and
  \item [(ii)] for all $i \in A^*_n$, there exists $i_0 \in A^*_0, ..., i_{n-1} \in A^*_{n-1}$ such that $i_0 \lhd i_1 \lhd ... \lhd i_{n-1} \lhd i$.
 \end{itemize}
 Let $A^*_0 = (A_0)^0$, which satisfies (i) by assumption and (ii) vacuously.  Now, suppose that $A^*_{n-1}$ is constructed.  For each $X \subseteq A^*_{n-1}$ with $|X| = N+1$ and $Y \subseteq A_n^{n (\mathrm{mod}\ 2)}$ with $|Y| = N+1$, we claim that there exists $i \in X$ and $j \in Y$ such that $i \lhd j$.  If not, then $i \nlhd j$ for all such $i, j$.  However, since $X \subseteq A^*_{n-1} \subseteq P^{n+1 (\mathrm{mod}\ 2)}$ and $Y \subseteq A_n^{n (\mathrm{mod}\ 2)}$, we see that $i \neq j$.  Furthermore, since $A_{n-1} \lhd A_n$, $j \nlhd i$.  Therefore $X \cup Y$ is an antichain.  However, this contradicts Definition \ref{Defn_IndiscernibleColoring} (i).  Therefore, choosing $i_0 \in X$ and $j_0 \in Y$ such that $i_0 \lhd j_0$, we consider now $(N+1)$-element subsets of $A^*_{n-1} - \{ i_0 \}$ and $(N+1)$-element subsets of $A_n^{n (\mathrm{mod}\ 2)} - \{ j_0 \}$.  Continuing in this manner, we see that there exists $A^*_n \subseteq A_n^{n (\mathrm{mod}\ 2)}$ such that each $i \in A^*_n$ is $\lhd$-below some element of $A^*_{n-1}$ and $|A^*_n| > (2N+1-n)(N+1)$.  Thus $A^*_n$ satisfies conditions (i) and (ii), as desired.

 Finally, consider $A^*_{2N+1}$.  By (i), $|A^*_{2N+1}| > 0$, so it is, in particular, non-empty.  Fix $i_{2N+1} \in A^*_{2N+1}$.  By condition (ii), there exists $i_0 \in A^*_0, ..., i_{2N} \in A^*_{2N}$ such that $i_0 \lhd i_1 \lhd ... \lhd i_{2N} \lhd i_{2N+1}$.  However, for each $\ell < 2N+2$, since $i_\ell \in A^*_\ell \subseteq P^{\ell (\mathrm{mod}\ 2)}$, $f(i_\ell) \equiv \ell (\mathrm{mod}\ 2)$.  This contradicts Definition \ref{Defn_IndiscernibleColoring} (ii).  Therefore, $K < 2N+2$, as desired.
\end{proof}

\begin{thm}[$N$-Indiscernible Coloring Decomposition Theorem]\label{Thm_PosetDefinability}
 Let $M = (2N+1)(N+1)$.  There exists $A_0 \lhd ... \lhd A_{K-1}$ for $K \le 2N+2$ maximal antichains of $P$ and $C_{n,\ell} \subseteq [A_{n-1}, A_n)$ for $\ell < M$ chains of $P$ such that
 \[
  P^1 = \left( \bigcup_{n \equiv 0 (\mathrm{mod}\ 2)} [A_{n-1}, A_n) - \left( \bigcup_{\ell < M} C_{n,\ell} \right) \cup  \bigcup_{n \equiv 1 (\mathrm{mod}\ 2), \ell < M} C_{n,\ell} \right).
 \]
\end{thm}

\begin{proof}
 Use the maximal antichains $A_0, ..., A_{K-1} \subseteq P$ as given by Lemma \ref{Lem_IndiscBreakdown}.  Fix $n \le K$ and consider $P_n = \{ i \in [A_{n-1}, A_n) : f(i) \equiv n (\mathrm{mod}\ 2) \}$.  By the condition given in Lemma \ref{Lem_IndiscBreakdown}, for each antichain $A$ of $(P_n, \unlhd)$, $|A| \le M$.  Therefore, by Theorem \ref{Thm_BoundAntiChainBoundChain}, $(P_n, \unlhd)$ is the disjoint union of at most $M$ chains, say $C_{n,\ell}$.  That is, $P_n = \bigcup_{\ell < M} C_{n,\ell}$.  Therefore, for each $n$ and each $i \in [A_{n-1},A_n)$, $f(i) \equiv n (\mathrm{mod}\ 2)$ if and only if $i \in \bigcup_{\ell < M} C_{n,\ell}$.  The conclusion follows.
\end{proof}

There are two problems with this decomposition in terms of uniform definability.  For one, the antichains $A_n$ may be arbitrarily large, so checking if $i \in [A_{n-1}, A_n)$ could require arbitrarily much information from $A_{n-1}$ and $A_n$.  Another problem is that the chains $C_{n,\ell}$ may be arbitrarily large.  We address the problems in reverse order.  First, for any $i_0, i_1 \in P$, define $[i_0, i_1]_P = \{ i \in P : i_0 \unlhd i \unlhd i_1 \}$.

\begin{lem}\label{Lem_ChainFix}
 Fix $t < 2$ and suppose that $C \subseteq P^t$ is any chain.  There exists $i_0 \unlhd i'_0 \unlhd i_1 \unlhd i'_1 \unlhd ... \unlhd i_K \unlhd i'_K$ from $C$ for $K \le N$ such that
 \[
  C \subseteq \left( \bigcup_{n \le K} [i_n, i'_n]_P \right) \subseteq P^t.
 \]
\end{lem}

\begin{proof}
 Let $i_0$ be the minimal element of $C$ and $i'_K$ the maximal element of $C$.  For every $j \in [i_0, i_K]_P$ such that $f(j) = 1-t$, let $i^-_j \lhd i^+_j$ be from $C$ so that $i^-_j \lhd j \lhd i^+_j$ and $i^-_j$ is $\lhd$-maximal such and $i^+_j$ is $\lhd$-minimal such.  Fix $n \ge 0$ and suppose that $i_n$ is constructed.  Choose $j \in ([i_0, i_K]_P)^{1-t}$ such that $i_n \unlhd i^-_j$, $i^-_j$ is $\lhd$-minimal such, and $i^+_j$ is $\lhd$-minimal such.  Let $i'_n \in C$ be $\lhd$-maximal such that $i'_n \lhd i^+_j$ and let $i_{n+1} = i^+_j$.  If no such $j$ exists, then the construction terminates and set $K = n$.

 First, it is clear that $C \subseteq \bigcup_{n < K} [i_n, i'_n]_P$ as, for each $n < K$, $i'_n$ and $i_{n+1}$ are $C$-consecutive.  We claim that, $\bigcup_{n < K} [i_n, i'_n]_P \subseteq P^t$ and $K \le N$.  First, fix $j \in [i_n, i'_n]_P$.  If $j = i_n$ or $j = i'_n$, then $f(j) = t$, so we may assume $j \neq i_n, i'_n$.  If $f(j) = 1-t$, then we have $i_n \lhd j \lhd i'_n$.  Therefore $i_n \unlhd i^-_j$ and $i^+_j \unlhd i'_n$, contrary to construction.  Therefore, $f(j) = t$.  Secondly, suppose that $K > N$.  By construction, for any $n < N$, there exists $j_n \in P^{1-t}$ such that $i_n \lhd j_n \lhd i_{n+1}$.  Hence, we see that $i_0 \lhd j_0 \lhd i_1 \lhd j_1 \lhd ... \lhd i_N \lhd j_N$ contradicts Definition \ref{Defn_IndiscernibleColoring} (ii).  This gives the desired result.
\end{proof}

So the chains $C_{n,\ell}$ can be taken to be a union of at most $N+1$ closed intervals.  What about the arbitrarily large antichains?

\begin{lem}\label{Lem_AntichainFix}
 For all maximal antichains $A$ and both $t < 2$ such that $|A^{1-t}| \le N$, there exists $A_0 \subseteq A$ with $|A_0| \le 2N+1$, $J^- \subseteq P^{1-t}$ with $|J^-| \le N$, and $J^+ \subseteq P^{1-t}$ with $|J^+| \le N$ so that, for all $j \in P^{1-t}$,
 \begin{itemize}
  \item [(i)] $j \lhd A$ if and only if $j \lhd A_0$ or $j \unlhd J^-$, and
  \item [(ii)] $A \lhd j$ if and only if $A_0 \lhd j$ or $J^+ \unlhd j$.
 \end{itemize}
\end{lem}

\begin{proof}
 If $|A| \le 2N$, then set $A_0 = A$, $J^- = J^+ = \emptyset$.  So we may assume $|A| > 2N$ and $t = \Maj(A)$.  Fix any subset $A_0 \subseteq A$ such that $|A_0^t| = N+1$ and $A_0^{1-t} = A^{1-t}$.  Since $t = \Maj(A)$, $|A^{1-t}| \le N$.  Therefore, $|A_0| \le 2N+1$.  We now construct $J^-$ by induction as follows: Let $J^-_0 = \emptyset$.  Fix $n > 0$ and suppose that $J^-_{n-1}$ is constructed.  If there exists $j \in P^{1-t}$ such that $j \lhd A$ and $A_0 \cup J^-_{n-1} \cup \{ j \}$ is an antichain, then choose $j' \lhd A$ $\lhd$-maximal such that $j \unlhd j'$.  Then, $A_0 \cup J^-_{n-1} \cup \{ j' \}$ is clearly still an antichain (if $j'$ were below some $i \in A_0 \cup J^-_{n-1} \cup \{ j' \}$, then $j$ would be too by transitivity).  Let $J^-_n = J^-_{n-1} \cup \{ j' \}$ and continue.  The construction halts when there exist no such $j$ and we set $J^- = J^-_{n-1}$.  Construct $J^+$ similarly for $j \in P^{1-t}$ so that $A \lhd j$.

 We claim that this construction works and (each) halts in at most $N$ steps.  Since $A_0 \cup J^-_n$ is an antichain, $|A_0^t| = N+1$, and $J^- \subseteq P^{1-t}$, if $|J^-| > N$, then this contradicts Definition \ref{Defn_IndiscernibleColoring} (i).  Therefore, $|J^-| \le N$.  If $j \in P^{1-t}$ and $j \lhd A$, then either $j \lhd A_0$, or $j \unlhd J^-$ by construction of $J^-$ and similarly for $J^+$.
\end{proof}

This lemma implies that maximal antichains form a strong barrier for the non-majority color.  That is, the relationship of any $j \in D(A)^{1-t}$ to all of $[A, \infty)$ is determined by a set of size $\le 3N+1$.

\begin{cor}\label{Cor_AntichainBarrier}
 Fix $A$ a maximal antichain and $t < 2$ such that $|A^{1-t}| \le N$.  For $A_0$ and $J^-$ as in Lemma \ref{Lem_AntichainFix}, there exists a partition of $[A, \infty)$ into $X_I$ for $I \subseteq (A_0 \cup J^-)$ such that, for all $j \in P^{1-t}$ with $j \lhd A$, for all $i \in [A, \infty)$, $j \lhd i$ if and only if $i \in X_{\{ i' \in A_0 : j \lhd i' \} \cup \{ j' \in J^- : j \unlhd j' \}}$.
\end{cor}

\begin{proof}
 Fix $A'_0$ and $J^-$ given as in Lemma \ref{Lem_AntichainFix}.  For $i \in [A, \infty)$, put $i \in X_I$ if and only if $I = \{ i' \in A_0 \cup J^- : i' \unlhd i \}$.
\end{proof}

A similar result holds for $A_0 \cup J^+$ and $[-\infty, A]$ by symmetry.  We will use Theorem \ref{Thm_PosetDefinability} and the other tools of this section in the next section to prove Theorem \ref{Thm_UDTFIPOS}.

\section{General $\Delta$-Indiscernibility}\label{Section_GeneralIndiscn}


\subsection{Introduction}

Work in a complete theory $T$ in a language $L$ with monster model $\mathfrak{C}$.  Fix $\Delta(\overline{z}_0, ..., \overline{z}_n)$ any set of $L$-formulas where $\leng(\overline{z}_i) = \leng(\overline{z}_j)$ and let $P$ be an $S$-structure for some different language, $S$ (we will call this language $S$ the \emph{index language}).  Let $\langle \overline{b}_i : i \in P \rangle$ be a sequence of elements from $\mathfrak{C}^{\leng(\overline{z}_0)}$ indexed by $P$.

\begin{defn}[General Indiscernibility]\label{Defn_GenIndiscSequence}
 The sequence $\langle \overline{b}_i : i \in P \rangle$ is \emph{$\Delta$-indiscernible} (with respect to the $S$-structure $P$) if, for all $i_0, ..., i_n \in P$ distinct and all $j_0, ..., j_n \in P$ distinct such that $\qftp_S(i_0, ..., i_n) = \qftp_S(j_0, ..., j_n)$ (i.e. there exists a partial $S$-elementary map $f$ so that $f(i_k) = j_k$ for all $k \le n$),
 \[
  \tp_\Delta(\overline{b}_{i_0}, ..., \overline{b}_{i_n}) = \tp_\Delta(\overline{b}_{j_0}, ..., \overline{b}_{j_n}).
 \]
 If we drop the $\Delta$, then we mean that $\langle \overline{b}_i : i \in P \rangle$ is $\Delta$-indiscernible for all appropriate $\Delta$.
\end{defn}

In this section, we will be interested in the case where $\Delta$ is finite, $S = \{ \unlhd \}$, and $P$ is a partial order.  In the case where $P$ is a linear order, Definition \ref{Defn_GenIndiscSequence} is the usual definition of a $\Delta$-indiscernible sequence.  When $P$ is completely unordered, Definition \ref{Defn_GenIndiscSequence} is the usual definition of a $\Delta$-indiscernible set.  Given $(P; \unlhd)$ a partial order, a sequence $\langle \overline{b}_i : i \in P \rangle$ is $\Delta$-indiscernible if and only if, for all $i_0, ..., i_n \in P$ distinct and all $j_0, ..., j_n \in P$ distinct, if $i_k \unlhd i_\ell$ if and only if $j_k \unlhd j_\ell$ for all $k, \ell \le n$, then
\[
 \tp_\Delta(\overline{b}_{i_0}, ..., \overline{b}_{i_n}) = \tp_\Delta(\overline{b}_{j_0}, ..., \overline{b}_{j_n}).
\]

Suppose now that $\varphi(\overline{x}; \overline{y})$ is any dependent formula.  For any $n < \omega$, define
\begin{equation}\label{Eq_DeltaN}
 \Delta_{n,\varphi}(\overline{z}_0, ..., \overline{z}_n) = \left\{ \exists \overline{x} \left( \bigwedge_{i \le n} \varphi(\overline{x}; \overline{z}_i)^{s(i)} \right) : s \in {}^{n+1} 2 \right\}.
\end{equation}
Using the machinery of indiscernible sequences with this special set of formulas $\Delta_{n,\varphi}$, we aim to prove the following theorem:

\begin{thm}\label{Thm_NIPPOIndisc}
 The following are equivalent for a partitioned formula $\varphi(\overline{x}; \overline{y})$:
 \begin{itemize}
  \item [(i)] $\varphi$ is dependent.
  \item [(ii)] There exists $N, K, L < \omega$ and formulas $\psi_\ell(\overline{y}; \overline{z}_0, ..., \overline{z}_K)$ for $\ell < L$ such that, for all finite partial orders $(P; \unlhd)$, all $\Delta_{N,\varphi}$-indiscernible sequences $\langle \overline{b}_i : i \in P \rangle$, and all $p \in S_\varphi(\{ \overline{b}_i : i \in P \})$, there exists $\ell < L$ and $i_0, ..., i_K \in P$ such that, for all $j \in P$,
   \[
    \varphi(\overline{x}; \overline{b}_j) \in p(\overline{x}) \text{ if and only if } \models \psi_\ell(\overline{b}_j; \overline{b}_{i_0}, ..., \overline{b}_{i_K}).
   \]
 \end{itemize}
\end{thm}

As an immediate corollary, we get Theorem \ref{Thm_UDTFIPOS}.  That is, a partitioned formula $\varphi(\overline{x}; \overline{y})$ is dependent if and only if it has uniform definability of types over finite partial order indiscernibles.  This generalizes the result of Theorem 1.2 (ii) of \cite{Mypaper2}.

First, to show that (ii) implies (i), we only need to count types.  Suppose, by means of contradiction, that $\varphi$ is independent and (ii) holds.  Then, by Ramsey's Theorem, for any $m < \omega$, there exists $\langle \overline{b}_i : i \in P \rangle$ a $\Delta_{N,\varphi}$-indiscernible sequence, where $(P; \unlhd)$ is a linear order with $|P| = m$, such that the set $\{ \overline{b}_i : i \in P \}$ is $\varphi$-independent.  Therefore, the size of $S_\varphi(\{ \overline{b}_i : i \in P \})$ is exactly $2^m$.  However, since each type in $S_\varphi(\{ \overline{b}_i : i \in P \})$ is determined by $\ell < L$ and $i_0, ..., i_K \in P$, the number of $\varphi$-types over $\{ \overline{b}_i : i \in P \}$ is $\le L \cdot |P|^K = L \cdot m^K$.  Therefore, $2^m \le L \cdot m^K$.  However, our choice of $m$ was arbitrary (in particular, independent of $L$ and $K$).  This is a contradiction.

The converse is trickier to show, and will involve a detailed analysis of $\Delta_{N,\varphi}$-indiscernible sequences.  Suppose $\varphi$ is dependent and let $N = \ID(\varphi)$.  Let $\Delta = \Delta_{N,\varphi}$ as in \eqref{Eq_DeltaN} above.  We begin with a lemma for $\Delta$-indiscernible sequences indexed by partial orders.  The proof of this lemma is a simple modification of the proof of Theorem II.4.13 of \cite{Shelah}, but we include it here for completeness.

\begin{lem}\label{Lem_IndiscChainAntiChains}
 Fix $(P; \unlhd)$ a partial order, let $\langle \overline{b}_i : i \in P \rangle$ be a $\Delta$-indiscernible sequence, and fix any $\overline{a} \in \mathfrak{C}^{\leng(\overline{x})}$.  Let $f : P \rightarrow 2$ be defined by, for all $i \in P$, $f(i) = 1$ if and only if $\models \varphi(\overline{a}; \overline{b}_i)$.  Then $f$ is an $N$-indiscernible coloring of $P$.
\end{lem}

\begin{proof}
 (i): Suppose, by way of contradiction, that there exists $i_0^0$, $...$, $i_N^0$, $i_0^1$, $...$, $i_N^1 \in A$ distinct from some antichain $A \subseteq P$ such that $f(i_\ell^t) = t$ for all $t < 2$ and $k \le N$.  Then, for any $s \in {}^{N+1} 2$, the following formula is witnessed by $\overline{a}$:
 \[
  \models \exists \overline{x} \left( \bigwedge_{\ell \le N} \varphi(\overline{x}; \overline{b}_{i_\ell^{s(\ell)}})^{s(\ell)} \right).
 \]
 However, by $\Delta$-indiscernibility, we get that
 \[
  \models \exists \overline{x} \left( \bigwedge_{\ell \le N} \varphi(\overline{x}; \overline{b}_{i_\ell^0})^{s(\ell)} \right).
 \]
 Since this holds for all $s \in {}^{N+1} 2$, we get that $\{ \overline{b}_{i_\ell^0} : \ell \le N \}$ is a $\varphi$-independent set of size $N+1$, contrary to the fact that $\ID(\varphi) = N$.
 
 (ii): Suppose, by way of contradiction, that we have $i_0 \lhd ... \lhd i_{2N+1}$ such that $f(i_\ell) \neq f(i_{\ell+1})$ for all $\ell < 2N+1$.  Without loss of generality, suppose $f(\overline{b}_{i_0}) = 0$.  For any $s \in {}^{N+1} 2$, as witnessed by $\overline{a}$, we have that
 \[
  \models \exists \overline{x} \left( \bigwedge_{\ell \le N} \varphi(\overline{x}; \overline{b}_{i_{2\ell + s(\ell)}})^{s(\ell)} \right).
 \]
 By $\Delta$-indiscernibility, we get that
 \[
  \models \exists \overline{x} \left( \bigwedge_{\ell \le N} \varphi(\overline{x}; \overline{b}_{i_{2\ell}})^{s(\ell)} \right).
 \]
 Again, this yields a contradiction.
\end{proof}

By Theorem \ref{Thm_PosetDefinability} and the tools of Section \ref{Section_PO}, to prove the remainder of Theorem \ref{Thm_NIPPOIndisc}, it suffices to show that the ordering of $P$ is $L$-definable.  For the remainder of this section, fix $(P; \unlhd)$ a finite partial order, $\langle \overline{b}_i : i \in P \rangle$ a $\Delta$-indiscernible sequence with respect to $P$, and $\overline{a} \in \mathfrak{C}^{\leng(\overline{x})}$.  As in the previous section, for any $X \subseteq P$ and $t < 2$, define
\[
 X^t = \{ i \in X : \models \varphi(\overline{a}; \overline{b}_i)^t \}.
\]

Our method for defining types will be to use Theorem \ref{Thm_PosetDefinability} to decompose $P$, then use uniform definitions for handling the various pieces.  For large subsets $X \subseteq P$, we cannot hope to get exact definitions for which $\overline{b}_i$ are such that $i \in X$ with a bounded number of parameters.  Instead, we will focus on ``rough definitions.''  Fix $t < 2$ and $X \subseteq P^t$.  We say that $X$ is \emph{roughly definable} if there exists $\gamma_X(\overline{y})$ uniform over boundedly many elements of $\{ \overline{b}_i : i \in P \}$ such that
\begin{itemize}
 \item [(i)] For all $i \in X$, $\models \gamma_X(\overline{b}_i)$, and
 \item [(ii)] For all $i \in P$, if $\models \gamma_X(\overline{b}_i)$, then $i \in P^t$.
\end{itemize}
If we can break up, for some $t < 2$, $P^t$ into a bounded number of subsets $X_0, ..., X_n$, each of which is roughly definable, then we can uniformly define which $i \in P^t$, hence develop a uniform definition of finite $\varphi$-types.  This will be the goal of the remainder of this section.

\subsection{Homogeneous Sets}

One useful tool will be homogeneity.

\begin{defn}\label{Defn_Homogeneous}
 We say that $X \subseteq P$ is \emph{homogeneous} (with respect to the $\Delta$-indiscernible sequence $\langle \overline{b}_i : i \in P \rangle$) if, for all $i_0, ..., i_N \in X$ distinct and all $j_0, ..., j_N \in X$ distinct,
 \[
  \tp_\Delta(\overline{b}_{i_0}, ..., \overline{b}_{i_N}) = \tp_\Delta(\overline{b}_{j_0}, ..., \overline{b}_{j_N}).
 \]
 That is, $\langle \overline{b}_i : i \in X \rangle$ is $\Delta$-indiscernible over the empty structure on $X$.
\end{defn}

For example, any antichain $A \subseteq P$ is homogeneous.  In the next few lemmas, we will show how to define large homogeneous subsets of $P$.  The following lemma is shown exactly as Lemma \ref{Lem_IndiscChainAntiChains} (i), noting that the only fact we used about $A$ was that it was homogeneous:

\begin{lem}\label{Lem_HomogenousColoring}
 For any $X \subseteq P$ homogeneous, there exists $t < 2$ such that $|X^t| \le N$.
\end{lem}

So the point now will be to start with some homogeneous set $X$ with a majority color $t < 2$.  Then, add on elements of $P^{1-t}$, preserving homogeneity, until this is no longer possible.  By Lemma \ref{Lem_HomogenousColoring}, we can add no more than $N$ elements from $P^{1-t}$ while still preserving homogeneity.  However, we will need to insure that the homogeneity of a large set $X$ is determined by a bounded subset $X_0 \subseteq X$.  We accomplish this by using Lemma \ref{Lem_AntichainFix}.

Fix $A \unlhd A'$ two maximal antichains and $t < 2$ so that $|A^t| > N$ and $|(A')^t| > N$.  Let $A_0$ and $J^-$ be given by Lemma \ref{Lem_AntichainFix} for $A$ and let $A'_0$ and $J^+$ be given by Lemma \ref{Lem_AntichainFix} for $A'$.  Let $J_0 = A_0 \cup J^-$ and $J_1 = A'_0 \cup J^+$ (so $|J_0| \le 2N+1$ and $|J_1| \le 2N+1$).  Therefore, by Lemma \ref{Lem_AntichainFix}, for all $j \in P^{1-t} - [A,A']$, $j \unlhd J_0$ or $J_1 \unlhd j$.  We now partition $[A,A']$ according to how it relates to $J_0$ and $J_1$ (similarly to Corollary \ref{Cor_AntichainBarrier}).  For all $J \subseteq J_0$ and $J' \subseteq J_1$, define
\[
 X_{J,J'} = \{ i \in [A,A'] : (\forall j \in J_0)(j \unlhd i \leftrightarrow j \in J) \wedge (\forall j' \in J_1)(i \unlhd j' \leftrightarrow j' \in J') \}.
\]

\begin{lem}\label{Lem_HomogeneityofRegions}
 Fix any two antichains $A_0, A_1 \in [A,A']$ and suppose that, for all $J \subseteq J_0$ and $J' \subseteq J_1$, either
 \begin{itemize}
  \item [(i)] $|A_0 \cap X_{J,J'}| = |A_1 \cap X_{J,J'}|$, or
  \item [(ii)] $|A_0 \cap X_{J,J'}| > N$ and $|A_1 \cap X_{J,J'}| > N$.
 \end{itemize}
 Then, for any $I_0 \subseteq P^{1-t} - [A, A']$, $A_0 \cup I_0$ is homogeneous if and only if $A_1 \cup I_0$ is homogeneous.
\end{lem}

\begin{proof}
 Fix $A_0$ and $A_1$ as above and $I_0 \subseteq P^{1-t} - [A, A']$.  Since the conditions are symmetric, suppose that $A_0 \cup I_0$ is homogeneous and we show that $A_1 \cup I_0$ is homogeneous.  Fix any $i_0, ..., i_N \in A_1 \cup I_0$ distinct and we will define $i'_\ell \in A_0 \cup I_0$ inductively on $\ell \le N$ so that the map $i_\ell \mapsto i'_\ell$ is an isomorphism of $S$-substructures.  First, if $i_\ell \in I_0$, then set $i'_\ell = i_\ell$.  Otherwise, fix $J, J'$ so that $i_\ell \in X_{J,J'}$ and choose $i'_\ell \in A_0 \cap X_{J,J'} - \{ i'_0, ..., i'_{\ell-1} \}$.  This exists by assumptions (i) or (ii).  Since $A_0$ and $A_1$ are both antichains, there is no relationship amongst the elements there.  For any $j \in I_0$, $j \lhd i_\ell$ (for $i_\ell \in X_{J,J'}$) if and only if $(\forall j' \in J_0)(j \unlhd j' \leftrightarrow j' \in J)$ for $j \unlhd J_0$ and likewise for $J_1 \unlhd j$ and $J'$.  But this holds if and only if $j \lhd i'_\ell$ as $i_\ell$ and $i'_\ell$ belong to the same $X_{J,J'}$.  Therefore, $i_\ell \mapsto i'_\ell$ is an isomorphism and, by $\Delta$-indiscernibility,
 \[
  \tp_\Delta(\overline{b}_{i_0}, ..., \overline{b}_{i_N}) = \tp_\Delta(\overline{b}_{i'_0}, ..., \overline{b}_{i'_N}).
 \]
 Since $A_0 \cup I_0$ is homogeneous, this implies that $A_1 \cup I_0$ is homogeneous.
\end{proof}

\begin{cor}\label{Cor_BoundedAddition1}
 Fix $A$ a maximal antichain and $t < 2$ such that $|A^{1-t}| \le N$.  There exists $A_0 \subseteq A$, with $|A_0| \le (N+1)2^{(2N+2)}$, such that, for any $I_0 \subseteq P^{1-t}$, $A \cup I_0$ is homogeneous if and only if $A_0 \cup I_0$ is homogeneous.
\end{cor}

\begin{proof}
 If $|A| \le 2N$, set $A_0 = A$ and we are done.  So we may assume $|A| > 2N$ and $|A^t| > N$.  Use Lemma \ref{Lem_HomogeneityofRegions} on $A = [A,A]$.  Then indeed any subset of $A = [A,A]$ is an antichain.  For each $J$, $J'$, choose $A_{J,J'}$ maximal in $X_{J,J'}$ such that $|A_{J,J'}| \le N$.  Then, taking $A_0 = \bigcup_{J,J'} A_{J,J'}$ suffices.
\end{proof}

\begin{lem}\label{Lem_DefiningAntichain}
 If $t < 2$ and $A \subseteq P^t$ is an antichain, then there exists a uniform formula $\gamma_A(\overline{y})$, over at most $N+(N+1)2^{(2N+2)}$ elements of $\{ \overline{b}_i : i \in P \}$, such that
 \begin{itemize}
  \item [(i)] for all $i \in A$, $\models \gamma_A(\overline{b}_i)$, and
  \item [(ii)] for all $i \in P$, if $\models \gamma_A(\overline{b}_i)$, then $i \in P^t$.
 \end{itemize}
 That is, antichains are roughly definable.
\end{lem}

\begin{proof}
 If $|A| \le N$, set $\gamma_A(\overline{y}) = \bigvee_{i \in A} \overline{y} = \overline{b}_i$.  So we may assume $|A| > N$, hence we can expand it to a maximal antichain $A' \subseteq P$ where $\Maj(A') = t$.  Fix $A_0$ as in Corollary \ref{Cor_BoundedAddition1} and choose $I_0 \subseteq P^{1-t}$ so that $A_0 \cup I_0$ is homogeneous and $I_0$ is maximal such.  By Lemma \ref{Lem_HomogenousColoring}, $|I_0| \le N$.

 Now, for any $j \in (P^{1-t} - (A_0 \cup I_0))$, since $A_0 \cup I_0 \cup \{ j \}$ is not homogeneous, there exists $i_0, ..., i_{N-1}, i_N \in A_0 \cup I_0$ distinct and $i'_0, ..., i'_{N-1} \in A_0 \cup I_0$ distinct such that
 \[
  \tp_\Delta(\overline{b}_{i_0}, ..., \overline{b}_{i_N}) \neq \tp_\Delta(\overline{b}_{i'_0}, ..., \overline{b}_{i'_{N-1}}, \overline{b}_j).
 \]
 Therefore, there exists $\delta \in \pm \Delta$ so that
 \[
  \models \delta(\overline{b}_{i_0}, ..., \overline{b}_{i_N}) \wedge \neg \delta(\overline{b}_{i'_0}, ..., \overline{b}_{i'_{N-1}}, \overline{b}_j).
 \]
 However, fix any $i \in A' - A_0$.  Then, by Corollary \ref{Cor_BoundedAddition1},
 \[
  \tp_\Delta(\overline{b}_{i_0}, ..., \overline{b}_{i_N}) = \tp_\Delta(\overline{b}_{i'_0}, ..., \overline{b}_{i'_{N-1}}, \overline{b}_i),
 \]
 hence $\models \delta(\overline{b}_{i'_0}, ..., \overline{b}_{i'_{N-1}}, \overline{b}_i)$.  Thus, the formula $\delta(\overline{b}_{i'_0}, ..., \overline{b}_{i'_{N-1}}, \overline{y})$ distinguishes $j$ from all $i \in A - A_0$.  Let
 \begin{align*}
  \gamma'(\overline{y}) = \bigwedge \Bigl\{ & \delta(\overline{b}_{i_0}, ..., \overline{b}_{i_{N-1}}, \overline{y}) : \delta \in \pm \Delta, i_0, ..., i_{N-1} \in A_0 \cup I_0, \\ & \models \delta(\overline{b}_{i_0}, ..., \overline{b}_{i_{N-1}}, \overline{b}_{i^*}) \text{ for some } i^* \in (A_0 \cup I_0) - \{ i_0, ..., i_{N-1} \} \Bigr\}.
 \end{align*}
 Finally, let
 \[
  \gamma_A(\overline{y}) = \left( \gamma'(\overline{y}) \wedge \bigwedge_{i \in A_0^{1-t} \cup I_0} \overline{y} \neq \overline{b}_i \right) \vee \bigvee_{i \in A_0^t} \overline{y} = \overline{b}_i.
 \]
 Then, for all $i \in A$, either $i \in A_0^t$ and clearly $\models \gamma_A(\overline{b}_i)$, or $i \in A - A_0 \subseteq A' - A_0$, in which case $\models \gamma'(\overline{b}_i)$ by construction.  Therefore, condition (i) holds.  Similarly, if $j \in P^{1-t}$, then either $j \in A_0^{1-t} \cup I_0$ and clearly $\models \neg \gamma_A(\overline{b}_j)$, or $j \notin A_0^{1-t} \cup I_0$, in which case $\models \neg \gamma'(\overline{b}_j)$ by construction.  This gives us condition (ii).
\end{proof}

In the next two subsections, we break the problem up into two cases depending on whether or not chains are homogeneous.  As in the previous section, let $M = (2N+1)(N+1)$.

\subsection{When Chains are Homogeneous}

For this subsection, we assume:

\noindent\textbf{Case 1.} For any $i_0 \lhd ... \lhd i_N$ from $P$, for all $\sigma \in S_{N+1}$ (the group of permutations on $N+1$), we have that
\[
 \tp_\Delta(\overline{b}_{i_0}, ..., \overline{b}_{i_N}) = \tp_\Delta(\overline{b}_{i_{\sigma(0)}}, ..., \overline{b}_{i_{\sigma(N)}}).
\]
That is, chains are homogeneous.  Thus, for any chain $C \subseteq P$, there exists $t < 2$ such that $|C^t| \le N$.

\begin{lem}\label{Lem_SimpleDecomposition}
 Under the assumption of Case 1, there exists $t < 2$ such that $P^t$ is a union of $\le M(2N+2)$ chains and $\le N(2N+2)$ antichains.
\end{lem}

\begin{proof}
 Let $A_0, ..., A_{K-1} \subseteq P$ as given by Theorem \ref{Thm_PosetDefinability} and let $P_n = [A_{n-1}, A_n)$ (where $A_{-1} = - \infty$ and $A_K = \infty$).  We already have that, for all $n \le K$, $P_n^{n (\mathrm{mod}\ 2)}$ is a union of $\le M$ chains by Theorem \ref{Thm_PosetDefinability}.  Fix $n \le K$ minimal such that $P_n^{(n+1) (\mathrm{mod}\ 2)}$ is not equal to a union of $\le N$ antichains and $\le M$ chains.  If $n = K$, then we can take $t = K (\mathrm{mod}\ 2)$ (as each $P_n^t$ is a union of $\le M$ chains or $\le M$ chains and $\le N$ antichains).  So we may assume that $n < K$.  We claim that $t = n (\mathrm{mod}\ 2)$ still works.

 Consider the antichains $\Lev_\ell^-(P_n^{1-t})$ for $\ell < N$ and let $P^*_n = P_n^{1-t} - ( \bigcup_{\ell < N} \Lev_\ell^-(P_n^{1-t}) )$.  Since $P_n^{1-t}$ is not the union of $\le N$ antichains (for example, $\Lev_\ell^-(P_n^{1-t})$ for $\ell < N$) and $\le M$ chains, $P^*_n$ cannot be the union of $\le M$ chains.  By Theorem \ref{Thm_BoundAntiChainBoundChain}, there exists $A \subseteq P^*_n$ an antichain with $|A| > M$.

 Now choose $A^* \subseteq P$ a maximal antichain with $|(A^*)^t| > M$, $A_n \lhd A^*$, and $A^*$ $\lhd$-maximal such.  This exists since $n < K$.  Thus, $(A^*, \infty)^t$ is a union of $\le M$ antichains and, for all $n' \le n$, $P_{n'}^t$ is a union of $\le M$ chains and $\le N$ antichains.  Thus, it suffices to check this condition for $[A_n, A^*]^t$.

 Let $P^{**} = [A_n, A^*]^t - \bigcup_{\ell < N} \Lev_\ell^+([A_n,A^*]^t)$.  We claim that $P^{**}$ has no antichain of size $> M$.  If it did, then fix $A' \subseteq P^{**}$ such an antichain.  By construction, $|A| > M$, $|A'| > M$, $A^{1-t} = A$, and $(A')^t = A'$.  Therefore, by Lemma \ref{Lem_IndiscChainAntiChains} (i), there exists $i \in A$ and $i' \in A'$ such that $i \lhd i'$ or $i' \lhd i$.  However, $i \lhd A_n$ and $A_n \unlhd i'$.  Therefore, $i \lhd i'$.  By the definition of levels, there exists $i_\ell \in \Lev_\ell^-(P_n^{1-t})$ such that $i_0 \lhd i_1 \lhd ... \lhd i_{N-1} \lhd i$ and there exists $i'_\ell \in \Lev_\ell^+([A_n,A^*]^t)$ such that $i' \lhd i'_0 \lhd i'_1 \lhd ... \lhd i'_{N-1}$.  This produces a chain $C$ so that $|C^0| > N$ and $|C^1| > N$, contrary to homogeneity of chains.

 Therefore, all antichains of $P^{**}$ have size $\le M$.  By Theorem \ref{Thm_BoundAntiChainBoundChain}, $P^{**}$ is the union of $\le M$ chains.  Therefore, $[A_n, A^*]^t$ is the union of $\le N$ chains and $\le M$ antichains.  Putting this together, we see that $P^t$ is the union of $\le M(2N+2)$ chains and $\le N(2N+2)$ antichains.
\end{proof}

\begin{lem}\label{Lem_DefiningChain}
 Under the assumption of Case 1, if $t < 2$ and $C \subseteq P^t$ is a chain, then there exists a uniform formula $\gamma_C(\overline{y})$, over at most $N \cdot (2(N+1)^2+N)$ elements of $\{ \overline{b}_i : i \in P \}$, such that
 \begin{itemize}
  \item [(i)] for all $i \in C$, $\models \gamma_C(\overline{b}_i)$, and
  \item [(ii)] for all $i \in P$, if $\models \gamma_C(\overline{b}_i)$, then $i \in P^t$.
 \end{itemize}
 That is, under Case 1, chains are roughly definable.
\end{lem}

\begin{proof}
 This is similar to the proof of Lemma \ref{Lem_DefiningAntichain}, noting that we are assuming chains are homogeneous.  Instead of having some small $C_0 \subseteq C$ which suffices to determine homogeneity for any $I_0 \subseteq P^{1-t}$ as in Corollary \ref{Cor_BoundedAddition1} for antichains, we will have to build $C_n$ as we go along.  By Lemma \ref{Lem_ChainFix}, we may assume that $C \subseteq [i,i']_P \subseteq P^t$ for some $i, i' \in C$ (this only decomposes $C$ into at most $N$ parts).  We may also assume that $|C| > N$, or else we just use $\gamma_C(\overline{y}) = \bigvee_{i \in C} \overline{y} = \overline{b}_i$.  For any $j \in P^{1-t}$, we have exactly one of three possibilities:
 \begin{itemize}
  \item [(i)] $\{ i, j \}$ is an antichain for all $i \in C$,
  \item [(ii)] there exists $i^-_j \in C$ such that $i^-_j \lhd j$ and $i^-_j$ is $\lhd$-maximal such, or
  \item [(iii)] there exists $i^+_j \in C$ such that $j \lhd i^+_j$ and $i^+_j$ is $\lhd$-minimal such.
 \end{itemize}
 In case (i), define $C'_j = \emptyset$.  In case (ii), define $C'_j$ to be the set of $N+1$ $C$-consecutive elements $\unlhd i^-_j$ and the $N+1$ $C$-consecutive elements $\rhd i^-_j$.  In case (iii), define $C'_j$ to be the set of $N+1$ $C$-consecutive elements $\lhd i^+_j$ and the $N+1$ $C$-consecutive elements $\unrhd i^+_j$.  Define $C_{-1}$ to be the set of $N+1$ $C$-initial elements and $N+1$ $C$-final elements.  Now we begin our construction of $j_0, j_1, ... \in P^{1-t}$ and $C_{-1} \subseteq C_0 \subseteq C_1 \subseteq ... \subseteq C$ (with $|C_n| \le 2(N+1)(n+1)$) as follows:

 Suppose that there exists $j \in P^{1-t} - \{ j_0, ..., j_{n-1} \}$ such that $C_{n-1} \cup \{ j_0, ..., j_{n-1}, j \}$ is homogeneous.  Let $j_n = j$ be any such and let $C_n = C_{n-1} \cup C'_j$.  We claim that, in fact, $C_n \cup \{ j_0, ..., j_n \}$ is homogeneous.  Given $i_0, ..., i_N \in C_n \cup \{ j_0, ..., j_n \}$, to get $i'_0, ..., i'_N \in C_{n-1} \cup \{ j_0, ..., j_n \}$ that are isomorphic (under $i_\ell \mapsto i'_\ell$), we fix the elements of $\{ j_0, ..., j_n \}$ and we push the elements inside $C'_{j_n}$ away to the nearest elements of $C_{n-1}$.  Then, the isomorphism yields
 \[
  \tp_\Delta(\overline{b}_{i_0}, ..., \overline{b}_{i_n}) = \tp_\Delta(\overline{b}_{i'_0}, ..., \overline{b}_{i'_n}),
 \]
 as desired.

 Therefore, this construction halts after at most $N$ steps, producing $I_0 = \{ j_0, ..., j_{K-1} \}$ and $C^* = C_K$ for $K \le N$.  Notice that, for any $j \in P^{1-t} - I_0$, $C \cup I \cup \{ j \}$ is homogeneous if and only if $C^* \cup I_0 \cup \{ j \}$ is homogeneous, which always fails by maximality of $I_0$.  The remainder of this proof follows exactly as the proof of Lemma \ref{Lem_DefiningAntichain}.
\end{proof}

We now prove Theorem \ref{Thm_NIPPOIndisc} (i) $\Rightarrow$ (ii) under Case 1.

By Lemma \ref{Lem_SimpleDecomposition}, there exists $t < 2$, chains $C_n$ for $n < M(2N+2)$ and antichains $A_m$ for $m < N(2N+2)$ so that $P^t = \bigcup_n C_n \cup \bigcup_m A_m$.  For each $n$, let $\gamma^*_n = \gamma_{C_n}$ given by Lemma \ref{Lem_DefiningChain}.  For each $m$, let $\gamma^{**}_m = \gamma_{A_m}$ given by Lemma \ref{Lem_DefiningAntichain}.  Then take
\[
 \gamma(\overline{y}) = \bigvee_n \gamma^*_n(\overline{y}) \vee \bigvee_m \gamma^{**}_m(\overline{y}).
\]
Then, for any $i \in P$, $\models \gamma(\overline{b}_i)$ if and only if $i \in P^t$ if and only if $\models \varphi(\overline{a}; \overline{b}_i)^t$.  Thus, the formula $\gamma^t$ defines the $\varphi$-type $p = \tp_\varphi(\overline{a} / \{ \overline{b}_i : i \in P \})$ in a uniform manner, as desired.

\subsection{When Chains are Not Homogeneous}

For this subsection, we assume:

\noindent\textbf{Case 2.} There exists $i_0 \lhd ... \lhd i_N$ from $P$ and $\sigma \in S_{N+1}$ such that
\[
 \tp_\Delta(\overline{b}_{i_0}, ..., \overline{b}_{i_N}) \neq \tp_\Delta(\overline{b}_{i_{\sigma(0)}}, ..., \overline{b}_{i_{\sigma(N)}}).
\]
That is, chains are not homogeneous.

\begin{defn}\label{Defn_OrderHomogeneous}
 Fix $X \subseteq P$ and $\le_X$ a linear order on $X$.  We say that $(X; \le_X)$ is \emph{order homogeneous} if, for any $i_0 <_X ... <_X i_N$ from $X$ and $j_0 <_X ... <_X j_N$ from $X$, we have that
 \[
  \tp_\Delta(\overline{b}_{i_0}, ..., \overline{b}_{i_N}) = \tp_\Delta(\overline{b}_{j_0}, ..., \overline{b}_{j_N}).
 \]
 In other words, $\langle \overline{b}_i : i \in X \rangle$ is $\Delta$-indiscernible with respect to $(X; \le_X)$.
\end{defn}

For example, for any chain $C' \subseteq P$, $(C'; \unlhd |_{C'})$ is order homogeneous.  The following lemma follows from Lemma \ref{Lem_IndiscChainAntiChains} on the partial order $(X; \le_X)$:

\begin{lem}\label{Lem_OrderHomogenousColoring}
 For any $X \subseteq P$ and $\le_X$ such that $(X; \le_X)$ is order homogeneous, there does not exist $i_0 <_X i_1 <_X ... <_X i_{2N+1}$ from $X$ and $s < 2$ such that $i_n \in X^s$ if and only if $n$ is even.
\end{lem}

Under the assumption of Case 2, for any $(X; \le_X)$ that is order homogeneous with $|X| \ge N+1$, $X$ is not homogeneous.  Therefore, by Lemma 1.3 of \cite{Mypaper2}, there exists $\ell < N$ and $\delta \in \pm \Delta$ such that, for all $i_0 <_X ... <_X i_N$ from $X$, we have that
\begin{equation}\label{Eq_OrderSensitive}
 \models \delta(\overline{b}_{i_0}, ..., \overline{b}_{i_N}) \wedge \neg \delta(\overline{b}_{i_0}, ..., \overline{b}_{i_{\ell-1}}, \overline{b}_{i_{\ell+1}}, \overline{b}_{i_\ell}, \overline{b}_{i_{\ell+2}}, ..., \overline{b}_{i_N}).
\end{equation}
That is, $\delta$ is order-sensitive at $\ell$.

\begin{lem}\label{Lem_IncomparableOrdering}
 Suppose that $i_0 \lhd ... \lhd i_N$ are from $P$ and $j \in P$ is such that one of the following conditions hold:
 \begin{itemize}
  \item [(i)] $\{ i_s, j \}$ is an antichain for all $s \le N$,
  \item [(ii)] for some $n \neq \ell$, $i_n \lhd j$ and $\{ i_s, j \}$ is an antichain for all $s > n$, or
  \item [(iii)] for some $n \neq \ell+1$, $j \lhd i_n$ and $\{ i_s, j \}$ is an antichain for all $s < n$.
 \end{itemize}
 Then the ordering $i_0 < i_1 < ... < i_\ell < j < i_{\ell+1} < ... < i_N$ is not order homogeneous.
\end{lem}

\begin{proof}
 Conditions (i), (ii), and (iii) insure that the map $i_s \mapsto i_s$ for all $s \neq \ell, \ell+1$, $j \mapsto j$, and $i_\ell \mapsto i_{\ell+1}$ is an isomorphism of $S$-substructures.  Therefore, by $\Delta$-indiscernibility,
 \begin{equation}\label{Eq_jswap}
  \tp_\Delta(\overline{b}_{i_0}, ..., \overline{b}_{i_{\ell-1}}, \overline{b}_j, \overline{b}_{i_{\ell+1}}, \overline{b}_{i_{\ell+2}}, ..., \overline{b}_{i_N}) = \tp_\Delta(\overline{b}_{i_0}, ..., \overline{b}_{i_{\ell-1}}, \overline{b}_j, \overline{b}_{i_\ell}, \overline{b}_{i_{\ell+2}}, ..., \overline{b}_{i_N}).
 \end{equation}
 However, for tautological reasons,
 \[
  \models \delta(\overline{b}_{i_0}, ..., \overline{b}_{i_{\ell-1}}, \overline{b}_j, \overline{b}_{i_{\ell+1}}, \overline{b}_{i_{\ell+2}}, ..., \overline{b}_{i_N}) \vee \neg \delta(\overline{b}_{i_0}, ..., \overline{b}_{i_{\ell-1}}, \overline{b}_j, \overline{b}_{i_{\ell+1}}, \overline{b}_{i_{\ell+2}}, ..., \overline{b}_{i_N}).
 \]
 By \eqref{Eq_jswap}, we get
 \[
  \models \delta(\overline{b}_{i_0}, ..., \overline{b}_{i_{\ell-1}}, \overline{b}_j, \overline{b}_{i_\ell}, \overline{b}_{i_{\ell+2}}, ..., \overline{b}_{i_N}) \vee \neg \delta(\overline{b}_{i_0}, ..., \overline{b}_{i_{\ell-1}}, \overline{b}_j, \overline{b}_{i_{\ell+1}}, \overline{b}_{i_{\ell+2}}, ..., \overline{b}_{i_N}).
 \]
 If $i_0 < i_1 < ... < i_\ell < j < i_{\ell+1} < ... < i_N$ were order homogeneous, this would imply that
 \[
  \models \delta(\overline{b}_{i_0}, ..., \overline{b}_{i_{\ell-1}}, \overline{b}_{i_{\ell+1}}, \overline{b}_{i_\ell}, \overline{b}_{i_{\ell+2}}, ..., \overline{b}_{i_N}) \vee \neg \delta(\overline{b}_{i_0}, ..., \overline{b}_{i_{\ell-1}}, \overline{b}_{i_\ell}, \overline{b}_{i_{\ell+1}}, \overline{b}_{i_{\ell+2}}, ..., \overline{b}_{i_N}).
 \]
 contrary to \eqref{Eq_OrderSensitive}.
\end{proof}

As a corollary, for any such $i_0 \lhd ... \lhd i_N$ and $j \in P$, if $i \in P$ is such that $i_\ell \lhd i \lhd i_{\ell+1}$, then the formula
\begin{equation}\label{Eq_Thetaseparate}
 \theta(\overline{y}) = \delta(\overline{b}_{i_0}, ..., \overline{b}_{i_{\ell-1}}, \overline{y}, \overline{b}_{i_{\ell+1}}, \overline{b}_{i_{\ell+2}}, ..., \overline{b}_{i_N}) \wedge \neg \delta(\overline{b}_{i_0}, ..., \overline{b}_{i_{\ell-1}}, \overline{y}, \overline{b}_{i_{\ell}}, \overline{b}_{i_{\ell+2}}, ..., \overline{b}_{i_N})
\end{equation}
holds for $\overline{b}_i$ and fails for $\overline{b}_j$.  We use this to prove the following result for chains:

\begin{lem}\label{Lem_DefiningChain2}
 Under the assumption of Case 2, if $t < 2$ and $C \subseteq P^t$ is a chain, then there exists a uniform formula $\gamma_C(\overline{y})$, over at most $N \cdot (2(N+1)^2+N)$ elements of $\{ \overline{b}_i : i \in P \}$, such that
 \begin{itemize}
  \item [(i)] for all $i \in C$, $\models \gamma_C(\overline{b}_i)$, and
  \item [(ii)] for all $i \in P$, if $\models \gamma_C(\overline{b}_i)$, then $i \in P^t$.
 \end{itemize}
 That is, under Case 2, chains are roughly definable.
\end{lem}

\begin{proof}
 As in the proof of Lemma \ref{Lem_DefiningChain}, we can assume $C \subseteq [i,i']_P \subseteq P^t$ for some $i, i' \in C$ and $|C| > N$.  Let $X_{-1} = C$ and $\le_{-1} = \unlhd |_C$.  Suppose that $n \ge 0$ and $(X_{n-1}; \le_{n-1})$ is constructed.  Suppose there exists $j \in P^{1-t}$ and $\le$ a linear order on $X_{n-1} \cup \{ j \}$ extending $\le_{n-1}$ such that
 \begin{itemize}
  \item [(a)] $(X_{n-1} \cup \{ j \}; \le)$ is order homogeneous,
  \item [(b)] for all $j^* \in (X_{n-1} - C) \cup \{ - \infty, \infty \}$, there exists at least $N+1$ elements from $C$ $\le$-between $j^*$ and $j$,
  \item [(c)] if $(\exists i \in C)(j \lhd i)$, then, for all $i \in C$, $j < i$ if and only if $j \lhd i$, and
  \item [(d)] if $(\exists i \in C)(i \lhd j)$, then, for all $i \in C$, $i < j$ if and only if $i \lhd j$.
 \end{itemize}
 Then set $X_n = X_{n-1} \cup \{ j \}$, and $\le_n = \le$.  If no such $j$ and $\le$ exists, then set $K = n$, $X^* = X_{K-1}$, and $\le^* = \le_{K-1}$ and the construction halts.  We claim that this construction halts after $N$ steps (i.e., $K \le N$).

 If not, fix $j_0 <^* ... <^* j_N$ from $X^* - C$ (this exists since $|X^* - C| = K > N$).  By construction, for each $n \le N$, there exists (at least $N+1$ many) $i_n$ such that $j_{n-1} \le i_n \le j_n$ (where $j_{-1} = - \infty$).  So we have $i_0 <^* j_0 <^* i_1 <^* ... <^* i_N <^* j_N$.  However, $i_n \in C \subseteq P^t$ and $j_n \in (X^* - C) \subseteq P^{1-t}$, so this contradicts Lemma \ref{Lem_OrderHomogenousColoring}.  So $K \le N$.  We now show how to define $\gamma_C$ from Lemma \ref{Lem_DefiningChain} using at most $2(N+1)^2+N$ elements from $C$.

 Let $C_{-1}^+$ be the $\lhd$-initial $N+1$ elements of $C$ and let $C_K^-$ be the $\lhd$-final $N+1$ elements of $C$.  Enumerate $X^* - C = \{ j_0, ..., j_{K-1} \}$ such that $j_0 <^* ... <^* j_{K-1}$ and let $C_n^-$ be the $\lhd$-final $N+1$ elements $i \in C$ so that $i <^* j_n$ and let $C_n^+$ be the $\lhd$-initial $N+1$ elements $i \in C$ so that $j_n <^* i$.  Finally, let $C_0 = C_{-1}^+ \cup \bigcup_{n < K} (C_n^- \cup C_n^+) \cup C_K^+$.  By construction, $|C_0| \le 2(N+1)^2$.  For each $n \le K$, consider
 \[
  G_n = \{ i \in C : (\exists i_{n-1} \in C_{n-1}^+, i_n \in C_n^-)(i_{n-1} \lhd i \lhd i_n) \},
 \]
 the gap of $C$ between $C_{n-1}^+$ and $C_n^-$.  It is clear that
 \[
  C = \bigcup_{n \le K} (C_{n-1}^+ \cup G_n \cup C_n^+),
 \]
 so, to define $\gamma_C$ over $C_0 \cup (X^* - C)$ (which has $\le 2(N+1)^2+N$ elements), we need only distinguish elements from $P^{1-t} - X^*$ and $G_n$ for each $n$.

 Fix $n \le K$ and $j \in P^{1-t} - X^*$.  As before, we have three cases to consider:
 \begin{itemize}
  \item [(i)] $\{ i, j \}$ is an antichain for all $i \in (C_{n-1}^+ \cup G_n \cup C_n^+)$,
  \item [(ii)] there exists $i \in (C_{n-1}^+ \cup G_n \cup C_n^+)$ such that $i \lhd j$, or
  \item [(iii)] there exists $i \in (C_{n-1}^+ \cup G_n \cup C_n^+)$ such that $j \lhd i$.
 \end{itemize}

 If (i) holds, then by Lemma \ref{Lem_IncomparableOrdering}, the formula $\theta(\overline{y})$ as in \eqref{Eq_Thetaseparate} separates $i \in G_n$ from $j$ as in case (i).

 If (ii) holds, then let $I^- = \{ i \in C : i \lhd j \}$ and let $I^+ = \{ i \in C : i \nlhd j \}$.  If $C_{n-1}^+ \not\subseteq I^-$, then fix $i_\ell \in C_{n-1}^+ - I^-$, $i_0 \lhd ... \lhd i_{\ell-1}$ from $C_{n-1}^+$ arbitrary such that $i_{\ell-1} \lhd i_\ell$, and $i_{\ell+1} \lhd ... \lhd i_N$ from $C_n^-$ arbitrary.  Again by Lemma \ref{Lem_IncomparableOrdering}, we see that the formula
 \[
  \delta(\overline{b}_{i_0}, ..., \overline{b}_{i_{\ell-1}}, \overline{b}_{i_{\ell+1}}, \overline{y}, \overline{b}_{i_{\ell+2}}, ..., \overline{b}_{i_N}) \vee \neg \delta(\overline{b}_{i_0}, ..., \overline{b}_{i_{\ell-1}}, \overline{b}_{i_\ell}, \overline{y}, \overline{b}_{i_{\ell+2}}, ..., \overline{b}_{i_N})
 \]
 holds of $\overline{b}_i$ for any $i \in G_n$ and fails for $\overline{b}_j$.  Therefore, we may assume $C_{n-1}^+ \subseteq I^-$.  Similarly, we may assume $C_n^- \subseteq I^+$.  Therefore, if we let $\le$ be the extension of $\le^*$ setting $i < j$ for all $i \in I^-$ and $j < i$ for all $i \in I^+$, conditions (b), (c), and (d) of the construction holds for $j$.  If (a) holds, then we contradict the fact that the construction halted, so we may assume that (a) fails.  Therefore, there exists $i_0 \lhd ... \lhd i_s$ from $(C_{n-1}^+ \cup G_n \cup C_n^-) \cap I^-$ and $i_{s+1} \lhd ... \lhd i_N$ from $(C_{n-1}^+ \cup G_n \cup C_n^-) \cap I^+$ so that
 \[
  \tp_\Delta(\overline{b}_{i_0}, ..., \overline{b}_{i_{s-1}}, \overline{b}_{i_s}, \overline{b}_{i_{s+1}}, ..., \overline{b}_{i_N}) \neq \tp_\Delta(\overline{b}_{i_0}, ..., \overline{b}_{i_{s-1}}, \overline{b}_j, \overline{b}_{i_{s+1}}, ..., \overline{b}_{i_N}).
 \]
 However, since $C_{n-1}^+ \subseteq I^-$ we may choose $i_0 \lhd ... \lhd i_{s-1}$ in $C_{n-1}^+$ (by $\Delta$-indiscernibility) and we may similarly choose $i_{s+1} \lhd ... \lhd i_N$ in $C_n^-$.  Also, this is witnessed by some $\delta \in \pm \Delta$.  Therefore, we have that
 \[
  \models \delta(\overline{b}_{i_0}, ..., \overline{b}_{i_{s-1}}, \overline{b}_i, \overline{b}_{i_{s+1}}, ..., \overline{b}_{i_N}) \wedge \neg \delta(\overline{b}_{i_0}, ..., \overline{b}_{i_{s-1}}, \overline{b}_j, \overline{b}_{i_{s+1}}, ..., \overline{b}_{i_N})
 \]
 for all $i \in G_n$.  Hence, this formula separates $i \in G_n$ and $j$ as in case (ii).  Case (iii) follows by symmetry.
\end{proof}

Now that we can roughly define chains, we have to deal with the space between two antichains.

\begin{lem}\label{Lem_DefiningAntichain2}
 Suppose that $A \unlhd A'$ two maximal antichains of $P$ and $t < 2$ are such that $|A^t| > N$ and $|(A')^t| > N$.  Then, there exists a uniform $\gamma_{A,A'}(\overline{y})$ over at most $2^{(2N+2)} \cdot (2N+2(N+1)^2)$ elements of $\{ \overline{b}_i : i \in P \}$ such that
 \begin{itemize}
  \item [(i)] for all $i \in [A,A']^t$, $\models \gamma_{A,A'}(\overline{b}_i)$, and
  \item [(ii)] for all $i \in P - [A,A']$, if $\models \gamma_{A,A'}(\overline{b}_i)$, then $i \in P^t$.
 \end{itemize}
 That is, $\gamma_{A,A'}$ roughly defines $[A,A']^t$ over $P - [A,A']$.
\end{lem}

\begin{proof}
 Fix $J_0$ and $J_1$ and $X_{J,J'}$ for each $J \subseteq J_0$ and $J' \subseteq J_1$ as in Lemma \ref{Lem_HomogeneityofRegions}.  Note that $|J_0| \le 2N+1$ and $|J_1| \le 2N+1$.  Fixing $J \subseteq J_0$ and $J' \subseteq J_1$, we now focus roughly defining $X_{J,J'}^t$ (and there are, at most, $2^{(2N+2)}$ of these sets).

 Let $A_s = \Lev_s^-(X_{J,J'}^t)$ for $s < \ell$, let $X^* = X_{J,J'}^t - \bigcup_{s \le \ell} A_s$, let $A_{N-s} = \Lev_s^+(X^*)$ for $s < N-\ell-1$, and let $X^{**} = X^* - \bigcup_{\ell + 1 < s \le N} A_s$.  Each of $A_s$ for $s \le N$ is roughly definable by $\gamma_{A_s}$ by Lemma \ref{Lem_DefiningAntichain}.  Therefore, if $X^{**} = \emptyset$, we are done, so suppose not.  Fix $A^-$ a $\lhd$-antichain of $X^{**}$ that is $\lhd$-minimal and $|A^-| = N+1$.  Similarly, fix $A^+$ a $\lhd$-antichain of $X^{**}$ that is $\lhd$-maximal and $|A^+| = N+1$.  If these do not exist, then $X^{**}$ is a union of $\le N$ chains, say $C_\ell$.  Use $\bigvee_\ell \gamma_{C_\ell}$ as in Lemma \ref{Lem_DefiningChain2} to roughly define $X^{**}$.  So we may assume $A^-$ and $A^+$ exist.  By choice of minimality of $A^-$, for any $i \in X^{**}$, either $A^- \cup \{ i \}$ is an antichain or $A^- \lhd i$.  The same holds for $A^+$ (except with the reverse ordering).

 Now fix $I^- \subseteq P^{1-t} - [A,A']$ maximal so that $A^- \cup I^-$ is homogeneous and fix $I^+ \subseteq P^{1-t} - [A,A']$ maximal so that $A^+ \cup I^+$ is homogeneous.  By Lemma \ref{Lem_HomogenousColoring}, $|I^-| \le N$ and $|I^+| \le N$.  For all $i \in A^-$, choose $\iota_{i,0} \lhd ... \lhd \iota_{i,\ell - 1} \lhd \iota_{i,\ell} = i$ from $X_{J,J'}^t$ (which exists by construction of $X^{**}$) and let $X^- = \{ \iota_{i,s} : i \in A^-, s \le \ell \}$.  Similarly construct $X^+$ with chains $i = \iota'_{i,\ell+1} \lhd ... \lhd \iota'_{i,N}$ for $i \in A^+$.  Let $X_0 = X^- \cup X^+ \cup I^- \cup I^+$ and notice that $|X_0| = 2N+(N+1)^2$.

 Now, fix any $j \in P^{1-t} - [A,A']$ and any $i \in X^{**}$.  We claim that $j$ does not have the same $\Delta$-type as $i$ over $X_0$.  That is, there exists $i_1, ..., i_N \in X_0$ distinct such that
 \[
  \tp_\Delta(\overline{b}_i; \overline{b}_{i_1}, ..., \overline{b}_{i_N}) \neq \tp_\Delta(\overline{b}_j; \overline{b}_{i_1}, ..., \overline{b}_{i_N}).
 \]
 First, if $A^- \cup \{ i \}$ is an antichain, then by Lemma \ref{Lem_HomogeneityofRegions}, $A^- \cup I^-$ is homogeneous if and only if $A^- \cup \{ i \} \cup I^-$ is homogeneous.  Therefore, if $j$ had the same $\Delta$-type as $i$ over $X_0$ ($X_0 \supseteq A^- \cup I^-$), then we would have that $A^- \cup \{ j \} \cup I^-$ is homogeneous, contrary to the maximality of $I^-$.  Thus they do not have the same $\Delta$-type.  We show this when $A^+ \cup \{ i \}$ is an antichain by symmetry.  So we may assume $A^- \lhd i \lhd A^+$.  Therefore, we have $i_0 \lhd ... \lhd i_\ell \lhd i \lhd i_{\ell+1} \lhd ... \lhd i_N$ with $i_0, ..., i_N \in X_0$ (by construction of $X_0$).  Hence, by Lemma \ref{Lem_IncomparableOrdering}, $j$ cannot have the same $\Delta$-type as $i$ over $X_0$.  Therefore, we can separate $i \in X^{**}$ from $j \in P^{1-t} - [A,A']$ with a formula over $X_0$.  Let
 \[
  \gamma_{X^{**}}(\overline{y}) = \bigvee_{i \in X^{**}} \bigwedge \{ \delta(\overline{y}, \overline{b}_{i_1}, ..., \overline{b}_{i_N}) : i_1, ..., i_N \in X_0, \models \delta(\overline{b}_i, \overline{b}_{i_1}, ..., \overline{b}_{i_N}) \}.
 \]
 Note that, \textit{a priori}, $\gamma_{X^{**}}$ ranges over arbitrarily many elements $i \in X^{**}$.  However, there are only boundedly many $\Delta$-types over $X_0$, so this is a uniform formula over $2N+(N+1)^2$ elements of $\{ \overline{b}_i : i \in P \}$.  By construction, for all $i \in X^{**}$, $\models \gamma_{X^{**}}(\overline{b}_i)$.  Furthermore, for all $j \in P^{1-t} - [A,A']$, $\models \neg \gamma_{X^{**}}(\overline{b}_j)$.  This gives the desired result.
\end{proof}

We now prove Theorem \ref{Thm_NIPPOIndisc} (i) $\Rightarrow$ (ii) under Case 2, completing the proof.

By Lemma \ref{Lem_IndiscBreakdown}, there exists $A_0 \lhd ... \lhd A_{K-1}$ for $K \le 2N+2$ maximal antichains of $P$ such that, for all $n \le K$ and all antichains $A \subseteq [A_{n-1}, A_n)$, $|A^{n (\mathrm{mod}\ 2)}| \le M$ (let $A_{-1} = - \infty$ and $A_K = \infty$).  For each $n \equiv 0 (\mathrm{mod}\ 2)$, let $A'_n \subseteq [A_{n-1}, A_n]$ be a maximal antichain of $P$ with $|(A'_n)^1| > M$ and choose $A'_n$ $\lhd$-maximal such.  For each $n \equiv 1 (\mathrm{mod}\ 2)$, for any antichain $A \subseteq (A'_{n-1}, A_n)^1$, $|A| \le M$ by construction.  Therefore, by Theorem \ref{Thm_BoundAntiChainBoundChain}, $(A'_{n-1}, A_n)^1 = \bigcup_{\ell < M} C_{n, \ell}$ for chains $C_{n,\ell} \subseteq P^1$.  Likewise, for $n \equiv 0 (\mathrm{mod}\ 2)$, $[A_{n-1}, A'_n]^0 = \bigcup_{\ell < M} C_{n, \ell}$ for chains $C_{n,\ell} \subseteq P^0$.  Let:
\[
 \psi(\overline{y}) = \bigvee_{n \equiv 0 (\mathrm{mod}\ 2)} \left( \gamma_{A_{n-1}, A'_n}(\overline{y}) \wedge \bigwedge_{\ell < M} \neg \gamma_{C_{n, \ell}}(\overline{y}) \right) \vee \bigvee_{n \equiv 1 (\mathrm{mod}\ 2), \ell < M} \gamma_{C_{n,\ell}}(\overline{y})
\]
for $\gamma_{A,A'}$ as in Lemma \ref{Lem_DefiningAntichain2} and $\gamma_C$ as in Lemma \ref{Lem_DefiningChain2}.  Then, $\models \varphi(\overline{a}; \overline{b}_i)$ if and only if $i \in P^1$ if and only if $i \in C_{n,\ell}$ for some $n \equiv 1 (\mathrm{mod}\ 2)$ and some $\ell < M$ or $i \in [A_{n-1}, A'_n] - \bigcup_{\ell < M} C_{n,\ell}$ for some $n \equiv 0 (\mathrm{mod}\ 2)$.  This holds if and only if $\models \psi(\overline{b}_i)$.  Since all of the formulas $\gamma_{A,A'}$ and $\gamma_C$ are uniform, $\psi$ is uniform.  This concludes the proof of Theorem \ref{Thm_NIPPOIndisc}.  As mentioned before, Theorem \ref{Thm_UDTFIPOS} follows as a corollary.

\section{Discussion}\label{Section_Discussion}


With Theorem \ref{Thm_UDTFIPOS} in hand, one is tempted to solve the general UDTFS Conjecture by the following means: Prove all finite sets can be made into a partial order indiscernible.  Unfortunately, there are simple examples to show that this is not true even when we assume that $\varphi$ has independence dimension $1$.

\begin{expl}\label{Expl_NotPartialOrderIndiscernible}
 Consider $X = 5 = \{ 0, 1, 2, 3, 4 \}$ and let
 \[
  Y = \{ \{ 0 \}, \{ 0, 1, 2 \}, \{ 2, 3, 4 \}, \{ 4 \} \},
 \]
 a subset of the powerset of $X$.  Let $R(x,y)$ be a binary relation that holds if and only if $x \in X$, $y \in Y$, and $x \in y$.  The relation $R(x;y)$ clearly has independence dimension $1$ but we claim that there is no partial order $\unlhd$ on $Y$ so that $\langle y : y \in Y \rangle$ is a $\Delta_{1,R}$-indiscernible sequence.  To see this, suppose there was such a $\unlhd$.  First notice that $\{ 0 \}$ and $\{ 0, 1, 2 \}$ cannot be an antichain since they are not homogeneous (i.e., $\tp_\Delta(\{ 0 \}, \{ 0, 1, 2 \}) \neq \tp_\Delta(\{ 0, 1, 2 \}, \{ 0 \})$ for $\Delta = \Delta_{1,R}$ since $\{ 0 \} \subseteq \{0, 1, 2 \}$ and not vice-versa).  Therefore, $\{ 0 \} \lhd \{ 0, 1, 2 \}$ or $\{ 0, 1, 2 \} \lhd \{ 0 \}$.  Now $\{ 0, 1, 2 \}$ and $\{ 2, 3, 4 \}$ must form an antichain as the $\Delta$-type of the pair is unequal to the $\Delta$-type of $( \{ 0 \}, \{0, 1, 2\})$ or $( \{ 0, 1, 2 \}, \{ 0 \})$.  Similarly, $\{ 0 \}$ and $\{ 4 \}$ must form an antichain, and therefore $\tp_\Delta( \{ 0, 1, 2 \}, \{ 2, 3, 4 \}) = \tp_\Delta( \{ 0 \}, \{ 4 \} )$.  However, this cannot hold; for example, the first pair intersect non-trivially while the second pair does not.  Therefore, this is a contradiction.
\end{expl}

The problem, of course, is distinguishing between the two types of incomparability when using the set-inclusion ordering.  Assuming independence dimension $\le 1$, if two sets are incomparable, then either they are disjoint or their union is the whole space.  This can be remedied by considering instead an index language $S = \{ \unlhd, E \}$ where $E$ is a binary relation symbol.  Then, one can use $E$ on incomparable elements to distinguish the two types of incomparability.  This leads to the following open question:

\begin{ques}
 Do all dependent formulas have uniform definability of types over indiscernible sequences indexed by finite $S$-structures $P$ so that $\unlhd^P$ is a partial order and $E^P$ is a symmetric binary relation on incomparable elements?
\end{ques}

An alternative solution is to only deal with formulas $\varphi$ of independence dimension $\le 1$ that are directed in the sense of \cite{Adler2008}.

Of course, as mentioned in the introduction, we would like to expand this notion of definability of types to even more general index structures.  For example:

\begin{ques}
 Do all dependent formulas have uniform definability of types over indiscernible sequences indexed by finite directed graphs?
\end{ques}

One problem with directed graphs is that, without transitivity, there is no notion of minimal elements.  All means of obtaining UDTFS both in this paper and in \cite{Mypaper2} use the fact that finite partial orders have minimal elements.  Thus it seems that an entirely new approach would be needed to answer the question for directed graphs.

\section*{Ackowledgements}\label{Sect_Ackowledgements}

We would like to thank Chris Laskowski for all of his helpful discussions with the author on the content of this paper.  The research for this paper was partially supported by Laskowski's NSF grants DMS-0600217 and 0901336.

\bibliographystyle{elsarticle-num}
\bibliography{Guingona004}

\end{document}